\def\bu{\bullet}
\def\marker{\>\hbox{${\vcenter{\vbox{
    \hrule height 0.4pt\hbox{\vrule width 0.4pt height 6pt
    \kern6pt\vrule width 0.4pt}\hrule height 0.4pt}}}$}\>}
\def\gpic#1{#1
     \smallskip\par\noindent{\centerline{\box\graph}} \medskip}
\documentclass[12pt]{article}

\usepackage{tikz}
\usetikzlibrary{shapes}
\usepackage{enumitem}
\usepackage{fullpage}

\usepackage{amsmath,amsfonts,amssymb,amsthm,graphics,amscd}

\newtheorem{theorem}{Theorem}[section]

\theoremstyle{definition}

\theoremstyle{plain}
\newtheorem{conjecture}[theorem]{Conjecture}

\newtheorem{question}[theorem]{Question}
\newtheorem{cor}[theorem]{Corollary}

\newtheorem{prop}[theorem]{Proposition}

\newcommand\p{\mathcal{P}}

\def\st{\colon\,}

\def\esub{\subseteq}

\def\FR{\frac}
\def\FL#1{\lfloor{#1}\rfloor}
\def\CL#1{\lceil{#1}\rceil}
\def\RR{{\mathbb R}}
\def\NN{{\mathbb N}}
\def\Knnt{K_{\CL{n/2},\FL{n/2}}}
\def\nosub{\not\esub}
\def\VEC#1#2#3{#1_{#2},\ldots,#1_{#3}}
\def\CH{\binom}
\def\SE#1#2#3{\sum_{#1=#2}^{#3}}

\author{Sarah J. Loeb\thanks{University of Illinois, Urbana, IL,
sloeb2@illinois.edu.  Research supported in part by a gift from Gene H. Golub
to the Mathematics Department of the University of Illinois.} and
Douglas B.~West\thanks{Zhejiang Normal University, Jinhua, China, and
University of Illinois, Urbana, IL, west@math.uiuc.edu.
Research supported by Recruitment Program of Foreign Experts,
1000 Talent Plan, State Administration of Foreign Experts Affairs, China.
}}

\title{Fractional and Circular Separation Dimension of Graphs}
\date{\today}
\begin{document}
\maketitle

\vspace{-2pc}

\begin{abstract}
The {\it separation dimension} of a graph $G$, written $\pi(G)$, is the
minimum number of linear orderings of $V(G)$ such that every two nonincident
edges are ``separated'' in some ordering, meaning that both endpoints of one
edge appear before both endpoints of the other.  We introduce the
{\it fractional separation dimension} $\pi_f(G)$, which is the minimum of
$a/b$ such that some $a$ linear orderings (repetition allowed) separate every
two nonincident edges at least $b$ times.

In contrast to separation dimension, fractional separation dimension is 
bounded: always $\pi_f(G)\le 3$, with equality if and only if $G$ contains
$K_4$.  There is no stronger bound even for bipartite graphs, since
$\pi_f(K_{m,m})=\pi_f(K_{m+1,m})=\FR{3m}{m+1}$.  We also compute $\pi_f(G)$ for
cycles and some complete tripartite graphs.  We show that $\pi_f(G)<\sqrt 2$
when $G$ is a tree and present a sequence of trees on which the value tends to
$4/3$.

Finally, we consider analogous problems for circular orderings, where pairs
of nonincident edges are separated unless their endpoints alternate.  Let
$\pi^\circ(G)$ be the number of circular orderings needed to separate all pairs
and $\pi_f^\circ(G)$ be the fractional version.  Among our results:
(1) $\pi^\circ(G)=1$ if and only $G$ is outerplanar.
(2) $\pi^\circ(G)\le2$ when $G$ is bipartite.
(3) $\pi^\circ(K_n)\ge\log_2\log_3(n-1)$.
(4) $\pi_f^\circ(G)\le\FR32$, with equality if and only if $K_4\esub G$.
(5) $\pi_f^\circ(K_{m,m})=\FR{3m-3}{2m-1}$.
\end{abstract}

\baselineskip 16pt

\section{Introduction}

A pair of nonincident edges in a graph $G$ is \emph{separated} by a linear
ordering of $V(G)$ if both vertices of one edge precede both vertices of the
other. The \emph{separation dimension} $\pi(G)$ of a graph $G$ is the minimum
number of vertex orderings that together separate every pair of nonincident
edges of $G$.  Graphs with at most three vertices have no such pairs, so
their separation dimension is $0$.  We therefore consider only graphs with
at least four vertices.

Introduced by Basavaraju (B), Chandran (C), Golumbic (G), Mathew (M), and
Rajendraprasad (R)~\cite{BCGMR} (full version in~\cite{BCGMR2}),
separation dimension is motivated by a geometric interpretation.
By viewing the orderings as giving coordinates for each vertex, the separation
dimension is the least $k$ such that the vertices of $G$ can be embedded in
$\mathbb{R}^k$ so that any two nonincident edges of $G$ are separated by a
hyperplane perpendicular to some coordinate axis (ties in a coordinate may be
broken arbitrarily.)

The upper bounds on $\pi(G)$ proved by BCGMR~\cite{BCGMR,BCGMR2} include
$\pi(G)\le3$ when $G$ is planar (sharp for $K_4$) and $\pi(G)\le4\log_{3/2}n$
when $G$ has $n$ vertices.  Since all pairs needing separation continue to need
separation when other edges are added, $\pi(G)\le \pi(H)$ when $G\esub H$;
call this fact {\it monotonicity}.  By monotonicity, the complete graph $K_n$
achieves the maximum among $n$-vertex graphs.  In general,
$\pi(G)\ge\log_2\FL{\FR12\omega(G)}$, where $\omega(G)=\max\{t\st K_t\esub G\}$.
This follows from the lower bound $\pi(K_{m,n}) \ge \log_2
\min\{m,n\}$~\cite{BCGMR,BCGMR2} and monotonicity.  Hence the growth rate of
$\pi(K_n)$ is logarithmic.  (For the {\it induced separation dimension}, 
introduced in GMR~\cite{GMR}, the only pairs needing separation are those whose
vertex sets induce exactly two edges, and monotonicity does not hold.)

BCMR~\cite{BCMR} proved $\pi(G) \in O(k \log \log n)$ for the $n$-vertex graphs
$G$ in which every subgraph has a vertex of degree at most $k$.  Letting $K'_n$
denote the graph produced from $K_n$ by subdividing every edge, they also showed
$\pi(K'_n)\in\Theta(\log\log n)$.  Thus separation dimension is unbounded
already on the family of graphs with average degree less than $4$.
In terms of the maximum vertex degree $\Delta(G)$, Alon and BCMR~\cite{ABCMR}
proved $\pi(G)\le 2^{9\log_2^* \Delta(G)}\Delta(G)$.  They also proved that
almost all $d$-regular graphs $G$ satisfy $\pi(G) \ge \lceil d/2 \rceil$.

Separation dimension is equivalently the restriction of another parameter to
the special case of line graphs.  The {\it boxicity} of a graph $G$, written
${\rm box}(G)$, is the least $k$ such that $G$ can be represented by assigning
each vertex an axis-parallel box in $\RR^k$ (that is, a cartesian product of
$k$ intervals) so that vertices are adjacent in $G$ if and only if their
assigned boxes intersect.  The initial paper~\cite{BCGMR} observed that
$\pi(G)={\rm box}(L(G))$, where $L(G)$ denotes the line graph of $G$ (including
when $G$ is a hypergraph).

We study a fractional version of separation dimension, using techniques that
apply for hypergraph covering problems in general.  Given a hypergraph $H$, the
{\it covering number} $\tau(H)$ is the minimum number of edges in $H$ whose
union is the full vertex set.  For separation dimension $\pi(G)$, the vertex
set of $H$ is the set of pairs of nonincident edges in $G$, and the edges of
$H$ are the sets of pairs separated by a single ordering of $V(G)$.  Many
minimization problems, including chromatic number, domination, poset dimension,
and so on, can be expressed in this way.

Given a hypergraph covering problem, the corresponding fractional problem
considers the difficulty of covering each vertex multiple times and measures
the average number of edges needed.  In particular, the {\it $t$-fold covering
number} $\tau_t(H)$ is the least number of edges in a list of edges (repetition
allowed) that covers each vertex at least $t$ times, and the {\it fractional
covering dimension} is $\liminf_t \tau_t(H)/t$.  In the special case that $H$
is the hypergraph associated with separation dimension, we obtain the
{\it $t$-fold separation dimension} $\pi_t(G)$ and the {\it fractional
separation dimension} $\pi_f(G)$.

Every list of $s$ edges in a hypergraph $H$ provides an upper bound on
$\tau_f(H)$; if it covers each vertex at least $t$ times, then it is called
an {\it $(s:t)$-covering}, and $\tau_f(H)\le s/t$.  This observation suffices
to determine the maximum value of the fractional separation dimension.  It is
bounded, even though the separation dimension is not (recall
$\pi(K_n)\ge \log\FL{n/2}$).

\begin{theorem}\label{bound}
$\pi_f(G)\le3$ for any graph $G$, with equality when $K_4\esub G$.
\end{theorem}
\begin{proof}
We may assume $|V(G)|\ge4$, since otherwise there are no separations to be
established and $\pi_f(G)\le\pi(G)=0$.  Now consider the set of all linear
orderings of $V(G)$.  For any two nonincident edges $ab$ and $cd$, consider
fixed positions of the other $n-4$ vertices in a linear ordering.  There are
$24$ such orderings, and eight of them separate $ab$ and $cd$.  Grouping the
orderings into such sets shows that $ab$ and $cd$ are separated $n!/3$ times.
Hence $\pi_f(G)\le 3$.

Now suppose $K_4\esub G$.  In a copy of $K_4$ there are three pairs of 
nonincident edges, and every linear ordering separates exactly one of them.
Hence to separate each at least $t$ times, $3t$ orderings must be used.
We obtain $\pi_t(G)\ge 3t$ for all $t$, so $\pi(G)\ge 3$.
\end{proof}

When $G$ is disconnected, the value on $G$ of $\pi_t$ for any $t$ (and hence
also the value of $\pi_f$) is just its maximum over the components of $G$.  We
therefore focus on connected graphs.  Also monotonicity holds for $\pi_f$
just as for $\pi$.

Fractional versions of hypergraph covering problems are discussed in the book
of Scheinerman and Ullman~\cite{SU}.  For every hypergraph covering problem,
the fractional covering number is the solution to the linear programming
relaxation of the integer linear program specifying $\tau(H)$.  One can use
this to express $\tau_f(G)$ in terms of a matrix game; we review this
transformation in Section~\ref{sec:equivalent} to make our presentation
self-contained.  The resulting game yields a strategy for proving results about
$\tau_f(H)$ and in particular about $\pi_f(G)$.

In Section~\ref{sec:small_girth}, we characterize the extremal graphs for
fractional separation dimension, proving that $\pi_f(G)=3$ only when
$K_4\esub G$.  No smaller bound can be given even for bipartite graphs;
we prove $\pi_f(K_{m,m})=\FR{3m}{m+1}$.

In Sections~\ref{sec:larger_girth} and~\ref{sec:trees} we consider sparser
graphs.  The {\it girth} of a graph is the minimum length of its cycles
(infinite if it has no cycles).  In Section~\ref{sec:larger_girth} we show
$\pi_f(C_n)=\FR n{n-2}$.  Also, the value is $\FR{30}{17}$ for the Petersen
graph and $\FR{28}{17}$ for the Heawood graph.  Although these results
suggested asking whether graphs with fixed girth could admit better bounds on
separation number, Alon~\cite{Alon} pointed out by using expander graphs that
large girth does not permit bounding $\pi_f(G)$ by any constant less than $3$
(see Section~\ref{sec:larger_girth}).  Nevertheless, we can still ask the
question for planar graphs.

\begin{question}
How large can $\pi_f(G)$ be when $G$ is a planar graph with girth at least
$g$?  
\end{question}

%

In Section~\ref{sec:trees}, we consider graphs without cycles.  We prove that
$\pi_f(G) < \sqrt{2}$ when $G$ is a tree.  The bound improves to
$\pi_f(T)\le\FR43$ for trees obtained from a subdivision of a star  by adding
any number of pendant edges at each leaf.  This is sharp; the tree with $4m+1$
vertices obtained by once subdividing every edge of $K_{1,2m}$ has diameter $4$
and fractional separation dimension $\frac{4m-2}{3m-1}$, which tends to $\FR43$.
We believe that the optimal bound for trees is strictly between $\FR43$ and
$\sqrt2$.

\begin{question} \label{question:trees}
What is the supremum of $\pi_f(G)$ when $G$ is a tree? 
\end{question}

In Section~\ref{multip}, we return to the realm of dense graphs with values of
$\pi_f$ near $3$.  We first compute $\pi_f(K_{m+1,qm})$.  The formula
yields $\pi_f(K_{m,r})< 3(1-\FR1{2m-1})$ for all $r$, so both parts of the
bipartite graph must grow to obtain a sequence of values approaching $3$.
In the special case $q=1$, we obtain $\pi_f(K_{m+1,m})=\FR{3m}{m+1}$.
Thus $\pi_f(\Knnt)=\FR{3m}{m+1}$, where $m=\FL{n/2}$.

\begin{conjecture} \label{conjecture:bipartite}
Among bipartite $n$-vertex graphs, $\pi_f$ is maximized by $\Knnt$.
\end{conjecture}

We also prove $\pi_f(K_{m,m,m}) = \FR{6m+2}{2m+1}$.  When $n=6r$, we thus have
$\pi_f(K_{2r,2r,2r})>\pi_f(K_{3r,3r})$.  Surprisingly, the value is larger for
a quite different complete tripartite graph.  Computer search verifies the
extreme among tripartite graphs up to $14$ vertices.  For $n=9$, there is an
anomaly, with $\pi_f(K_{3,3,3})>\pi_f(K_{1,4,4})$.

\begin{conjecture} \label{conjecture:tripartite}
For $n\ge10$, the $n$-vertex graph not containing $K_4$ that maximizes $\pi_f$
is $K_{1,\FL{(n-1)/2},\CL{(n-1)/2}}$.
\end{conjecture}

Since $\pi_f(G)$ is always rational, we ask

\begin{question}
Which rational numbers (between $1$ and $3$) occur as the 
fractional separation dimension of some graph?
\end{question}

Finally, in Section~\ref{sec:circular}, we consider the analogues of $\pi$ and
$\pi_f$ defined by using circular orderings of the vertices rather than linear
ones; we use the notation $\pi^\circ$ and $\pi_f^\circ$.  We show first that
$\pi^\circ(G)=1$ if and only if $G$ is outerplanar.  Surprisingly,
$\pi^\circ(K_{m,n})=2$ when $m,n\ge2$ and $mn>4$, but $\pi^\circ$ is 
unbounded, with $\pi^\circ(K_n)>\log_2\log_3(n-1)$.

For the fractional context, we prove $\pi_f^\circ(G)\le\FR32$ for all $G$,
with equality if and only if $K_4\esub G$.  Again no better bound holds
for bipartite graphs; we prove $\pi_f^\circ(K_{m,qm})=\FR{6(qm-1)}{4mq+q-3}$,
which tends to $\FR32$ as $m\to\infty$ when $q=1$.  It tends to $\FR{6m}{4m+1}$
when $q\to\infty$, so again both parts must grow to obtain a sequence on which
$\pi_f^\circ$ tends to $\FR32$.  The proof is different from the linear case.
The questions remaining are analogous to those for $\pi_f$.

\begin{question}
How large can $\pi_f^\circ$ be when $G$ is a planar graph with girth at
least $g$?  Which are the $n$-vertex graphs maximizing $\pi_f^\circ$ among
bipartite graphs and among those not containing $K_4$?  Which rational
numbers between $1$ and $\FR32$ occur?
\end{question}

\section{Fractional Covering and Matrix Games} \label{sec:equivalent}
Given a hypergaph $H$ with vertex set $V(H)$ and edge set $E(H)$, let
$E_v=\{e\in E(H)\st v\in e\}$ for $v\in V(H)$.  The covering number $\tau(H)$
is the solution to the integer linear program ``minimize $\sum_{e\in E(H)} x_e$
such that $x_e\in\{0,1\}$ for $e\in E(H)$ and $\sum_{e\in E_v} x_e\ge 1$ for
$v\in V(H)$.''  The linear programming relaxation replaces the constraint
$x_e\in\{0,1\}$ with $0\le x_e\le 1$.

It is well known (see Theorem 1.2.1 of~\cite{SU}) that the resulting solution
$\tau^*$ equals $\tau_f(H)$.  Multiplying the values in that solution by their
least common multiple $t$ yields a list of edges covering each vertex at least
$t$ times, and hence $\tau_f(H)\le \tau^*t/t$.  Similarly, normalizing an
$(s:t)$-covering yields $\tau^*\le s/t$.  Note that since the solution to a
linear program with integer constraints is always rational, always $\tau_f(H)$
is rational (when $H$ is finite).

A subsequent transformation to a matrix game yields a technique for proving
bounds on $\tau_f(H)$.  The constraint matrix $M$ for the linear program has
rows indexed by $E(H)$ and columns indexed by $V(H)$, with $M_{e,v}=1$ when
$v\in e$ and otherwise $M_{e,v}=0$.  In the resulting matrix game, the edge
player chooses a row $e$ and the vertex player chooses a column $v$, and the
outcome is $M_{e,v}$.  In playing the game repeatedly, each player uses a
strategy that is a probability distribution over the options, and then the
expected outcome is the probability that the chosen vertex is covered by the
chosen edge.  The edge or ``covering'' player wants to maximize this
probability; the vertex player wants to minimize it.

Using the probability distribution $x$ over the rows guarantees outcome at
least the smallest entry in $x^T M$, no matter what the vertex player does.
Hence the edge player seeks a probability distribution $x$ to maximize $t$ such
that $\sum_{e\in E_v}x_e\ge t$ for all $v\in V(H)$.  Dividing by $t$ turns this
into the linear programming formulation for $\tau_f(H)$, with the resulting
optimum being $1/t$.  This yields the following relationship.

\begin{prop}\label{game}
(Theorem~1.4.1 of~\cite{SU})
If $M$ is the covering matrix for a hypergraph $H$, then $\tau_f(H)=1/t$,
where $t$ is the value of the matrix game given by $M$.
\end{prop}

Just as any strategy $x$ for the edge player establishes 
$\min x^T M$ as a lower bound on the value, so any strategy $y$ for the vertex
player establishes $\max My$ as an upper bound.  The value is established by
providing strategies $x$ and $y$ so that these bounds are equal.
As noted in~\cite{SU}, such strategies always exist.

For fractional separation dimension, we thus obtain the {\it separation game}.
The rows correspond to vertex orderings and the columns to pairs of nonincident
edges.  The players are the {\it ordering player} and the {\it pair player},
respectively.  To prove $\pi_f(G)\le 1/t$, it suffices to find a distribution
for the ordering player such that each nonincident pair is separated with
probability at least $t$.  To prove $\pi_f(G)\ge 1/t$, it suffices to find a
distribution for the pair player such that for each ordering the probability
that the chosen pair is separated is at most $t$.
\looseness -1

The proof of Theorem~\ref{bound} can be phrased in this language.
By making all vertex orderings equally likely, the ordering player achieves
separation probability exactly $\FR13$ for each pair, yielding $\pi_f(G)\le 3$.
By playing the three nonincident pairs in a single copy of $K_4$ equally likely
and ignoring all other pairs, the pair player achieves separation probability
exactly $1/3$ against any ordering, yielding $\pi_f(G)\ge3$.

Another standard result about these games will be useful to us.
Let $\p$ denote the set of pairs of nonincident edges in a graph $G$.
Symmetry in $G$ greatly simplifies the task of finding an optimal strategy
for the pair player.

\begin{prop}\label{symmetry}
(Exercise 1.7.3 of~\cite{SU})
If, for any two pairs of nonincident edges in a graph $G$, some automorphism of
$G$ maps one to the other, then there is an optimal strategy for the pair
player in which all pairs in $\p$ are made equally likely.  In general, there
is an optimal strategy that is constant on orbits of the pairs under the
automorphism group of $G$.
\end{prop}
\begin{proof}
Consider an optimal strategy $y$, yielding $\max My=t$.  Automorphisms of $G$
induce permutations of the coordinates of $y$.  The entries in $My'$ for any
resulting strategy $y'$ are the same as in $My$.  Summing these vectors over
all permutations and dividing by the number of permutations yields a strategy
$y^*$ that is constant over orbits and satisfies $\max My^*\le t$.
\end{proof}

When there is an optimal strategy in which the pair player plays all pairs
in $\p$ equally, the value of the separation game is just the largest fraction
of $\p$ separated by any ordering.  For $\tau_f(H)$ in general,
Proposition 1.3.4\ in~\cite{SU} states this by saying that for a
vertex-transitive hypergraph $H$, always $\tau_f(H)=|V(H)|/r$, where $r$ is the
maximum size of an edge.  For separation dimension, this yields the following:

\begin{cor}\label{trans}
Let $G$ be a graph.  If for any two pairs of nonincident edges in $G$, there is
an automorphism of $G$ mapping one pair of edges to the other, then
$\tau_f(G)=q/r$, where $q$ is the number of nonincident pairs of edges in $G$
and $r$ is the maximum number of pairs separated by any vertex ordering. 
\end{cor}

\section{Characterizing the Extremal Graphs} \label{sec:small_girth}

When $K_4\nosub G$, we can separate $\pi_f(G)$ from $3$ by a function of $n$.

\begin{theorem} \label{theorem:K4free}
If $G$ is an $n$-vertex graph and $K_4 \not\subseteq G$, then
$\pi_f(G) \le 3\left(1 - \frac{12}{n^4}+O(\FR1{n^5})\right)$. 
\end{theorem}
\begin{proof}
Let $p=\FR13+\FR{4(n-4)!}{n!}$; note that $1/p$ has the form
$3\left(1-\frac{12}{n^4}+O(\FR1{n^5})\right)$.  It suffices to give a
probability distribution on the orderings of $V(G)$ such that each nonincident
pair of edges is separated with probability at least $p$.  We do this by
modifying the list of all orderings.

Choose any four vertices $a,b,c,d\in V(G)$.  For each ordering $\rho$
of the remaining $n-4$ vertices, $24$ orderings begin with $\{a,b,c,d\}$ and
end with $\rho$.  By symmetry, we may assume $ac\notin E(G)$.  Thus the
possible pairs of nonincident edges induced by $\{a,b,c,d\}$ are $\{ab,cd\}$
and $\{ad,bc\}$.  We increase the separation probability for these pairs, even
though these four edges need not all exist.

The pairs $\{ab,cd\}$ and $\{ad,bc\}$ are each separated eight times in the list
of $24$ orderings.  We replace these $24$ with another list of $24$ (that is,
the same total weight) that separate $\{ab,cd\}$ and $\{ad,bc\}$ each at least
twelve times, while any other pair of disjoint vertex pairs not involving
$\{a,c\}$ is separated at least eight times.  Since $\{a,b,c,d\}$ is arbitrary
and we do this for each $4$-set, the pairs $\{ab,cd\}$ and $\{ad,bc\}$ remain
separated at least eight times in all other groups of $24$ orderings.  Thus the
separation probability increases from $\FR13$ to at least $p$ for all pairs of
nonincident edges.

Use four orderings each that start with $abcd$ or $bcad$ and eight each that
start with $cdba$ or $adbc$, always followed by $\rho$.  By inspection, each of
$\{ab,cd\}$ and $\{ad,bc\}$ is separated twelve times in the list.  The number
of orderings that separate any pair of nonincident edges having at most two
vertices in $\{a,b,c,d\}$ does not change.

It remains only to check pairs with three vertices in this set, consisting of
one edge induced by this set and another edge with one endpoint in the set.
The induced edge is one of $\{ab,cd, bc, ad, bd\}$ (never $ac$), and the other
edge uses one of the remaining two vertices in $\{a,b,c,d\}$.  In each case,
the endpoints of the induced edge appear before the third vertex in at least
eight of the orderings in the new list of $24$; this completes the proof.
\end{proof}

For $n$-vertex graphs not containing $K_4$, Theorem~\ref{theorem:K4free}
separates $\pi_f(G)$ from $3$ by a small amount.  We believe that a much larger
separation also holds (Conjecture~\ref{conjecture:tripartite}).  Nevertheless,
we show next that even when $G$ is bipartite there is no upper bound less than
$3$.

\begin{theorem}\label{theorem:balanced bipartite}\label{kmm}
$\pi_f(K_{m,m}) = \FR{3m}{m+1}$ for $m\ge2$.
\end{theorem}

\begin{proof}
The pairs in $\p$ all lie in the same orbit under automorphisms of $K_{m,m}$,
so Corollary~\ref{trans} applies.  There are $2\CH m2^2$ pairs in $\p$ (played
equally by the pair player).  It suffices to show that the maximum number of
pairs separated by any ordering is $\FR{m+1}{3m}2\CH m2^2$.

Let the parts of $K_{m,m}$ be $X$ and $Y$. Let $\sigma$ be an ordering
$\VEC v1{2m}$ such that each pair $\{v_{2i-1},v_{2i}\}$ consists of one
vertex of $X$ and one vertex of $Y$.  The ordering player will in fact make
all such orderings equally likely.  It suffices to show that $\sigma$ separates
$\FR{m+1}{3m}2\CH m2^2$ pairs and that no ordering separates more.

By symmetry, we may index $X$ as $\VEC x1m$ and $Y$ as $\VEC y1m$ in order in
$\sigma$, so that $\{v_{2i-1},v_{2i}\}=\{x_i,y_i\}$ for $1\le i\le m$, though
$x_i$ and $y_i$ may appear in either order.  Consider an element of $\p$
separated by $\sigma$.  The vertices involved in the separation may use
two, three, or four indices among $1$ through $m$.

Pairs hitting $i,j,k,l$ with $i<j<k<l$ must be separating $x_iy_j$ or
$y_ix_j$ from $x_ky_l$ or $y_kx_l$.  Hence there are $4\CH m4$ such pairs.

Pairs hitting only $i,j,k$ with $i<j<k$ involve two vertices with the same
index.  If that index is $i$ or $k$, then there are two ways to complete the
edge pair.  However, if $x_j$ and $y_j$ are both used, then there is only one
way to choose from $\{x_i,y_i\}$ and from $\{x_k,y_k\}$ to complete a separated
pair, determined by the order of $x_j$ and $y_j$.  Hence there are $5\CH m3$
such pairs.

A separated pair hitting only $i$ and $j$ must be $\{x_iy_i,x_jy_j\}$.
Hence in total $\sigma$ separates $4\CH m4+5\CH m3+\CH m2$ pairs in $\p$.
In fact, this sum equals $\FR{m+1}{3m}2\CH m2^2$.

Now let $\sigma$ be an ordering not of the specified form.  By symmetry we
may again index $X$ as $\VEC x1m$ and $Y$ as $\VEC y1m$ in order in $\sigma$.
However, now some vertex precedes another vertex with a lesser index.  That is,
by symmetry we may assume that $y_j$ appears immediately before $x_i$ for some
$i$ and $j$ with $j>i$.

Form $\sigma'$ from $\sigma$ by interchanging the positions of $y_j$ and $x_i$.
Any pair separated by exactly one of $\sigma$ and $\sigma'$ has $x_i$ and $y_j$
as endpoints of the two distinct edges.  There are $(i-1)(m-j)$ such pairs in
$\sigma$ and $(j-1)(m-i)$ such pairs in $\sigma'$.  Since $m\ge2$ and $j>i$,
comparing these quantities shows that $\sigma'$ separates strictly more pairs
than $\sigma$.
\end{proof}

To prove that always $\pi_f(\Knnt)=\FR{3m}{m+1}$, where $m=\FL{n/2}$,
we need also to compute $\pi_f(K_{m+1,m})$.  We postpone this to
Section~\ref{multip}.  Note that the simple final expression arises
when we cancellation common factors in the numerator and denominator.
We would hope that such a simple formula has a simple direct
proof, but we have not found one.

\section{Graphs with Larger Girth} \label{sec:larger_girth}

Among sparser graphs, it is natural to think first about cycles.

\begin{prop} \label{prop:cycles}
$\pi_f(C_n) = \FR n{n-2}$, for $n \ge 4$.
\end{prop}

\begin{proof}
The ordering player uses the $n$ rotations of an $n$-vertex path along the 
cycle, equally likely.  Nonincident edges $e$ and $e'$ are separated unless
$e$ or $e'$ consists of the first and last vertex.  Hence any pair in $\p$
is separated with probability $\FR{n-2}n$.

Letting the vertices be $\VEC v1n$ in order along the cycle, the pair player
makes the pairs $\{v_{i-1}v_i,v_{i+1}v_{i+2}\}$ (modulo $n$) equally likely.
It suffices to show that any ordering separates at most $n-2$ of these pairs.
Otherwise, by symmetry some ordering $\sigma$ separates the $n-1$ of them
satisfying $2\le i\le n$.  By symmetry $\{v_1,v_2\}$ precedes $\{v_3,v_4\}$
in $\sigma$.  If $v_i$ precedes $v_{i+2}$, then separating $v_iv_{i+1}$ from
$v_{i+2}v_{i+3}$ requires $v_{i+1}$ to precede $v_{i+3}$.  Iterating this
argument yields $v_{n-2}$ before $v_n$ and $v_{n-1}$ before $v_1$ in $\sigma$.
Since $v_1$ precedes both $v_3$ and $v_4$, choosing the right one by parity
leads to $v_1$ preceding $v_1$, a contradiction.
\end{proof} 

Proposition~\ref{prop:cycles} suggests that $\pi_f$ decreases as girth
increases.  We compute the values for the smallest $3$-regular graphs of
girth $5$ and girth $6$.

\begin{prop}
The fractional separation dimension of the Petersen graph is $\FR{30}{17}$. 
\end{prop}
\begin{proof}
The $75$ pairs of nonincident edges in the Petersen graph fall into two orbits:
the $15$ pairs that occur as opposite edges on a 6-cycle (Type 1) and the
$60$ pairs that do not (Type 2).  Thus some optimal strategy for the pair
player will make Type 1 pairs equally likely and make Type 2 pairs equally
likely.  There exist orderings that separate nine Type 1 pairs and $34$ Type 2
pairs.  With the graph expressed as the disjointness graph of the $2$-element
subsets of $\{1,2,3,4,5\}$, such an ordering $\sigma$ is
$$
12,34,51,23,45,13,42,35,41,25
$$
Since $\FR{9}{15}>\FR{34}{60}$, making the $120$ orderings 
generated from $\sigma$ by permuting $\{1,2,3,4,5\}$ equally likely
establishes $\FR{30}{17}$ as an upper bound on the fractional separation
dimension.

Since $\FR{9}{15}>\FR{34}{60}$, the pair player establishes a matching lower
bound by playing only Type 2 pairs, equally likely, if no ordering separates
more than $34$ Type 2 pairs.  Computer search shows that this is true.
\end{proof}

\begin{prop}
The fractional separation dimension of the Heawood graph is $\FR{28}{17}$. 
\end{prop}

\begin{proof}
The $168$ pairs of nonincident edges in the Heawood graph fall into two orbits:
84 pairs that are opposite on a a 6-cycle (Type 1) and 84 pairs that have a
common incident edge (Type 2).  Some optimal strategy for the pair player will
make Type 1 pairs equally likely and make Type 2 pairs equally likely.
There exist orderings that separate 51 Type 2 pairs and 54 Type 1 pairs.
With the graph expressed as the incidence graph of the Fano plane, such an
ordering is
$$
1, 124, 4, 457, 5, 561, 6, 346, 3, 235, 2, 672, 7, 713
$$
Making the 336 orderings generated by automorphisms equally likely establishes
$\FR{28}{17}$ as an upper bound on the fractional separation dimension.

The pair player establishes a matching lower bound by playing only Type 2
pairs, equally likely, if no ordering separates more than $51$ Type 2 pairs.
Computer search (reduced by symmetries) shows that this is true.
\end{proof}


These small graphs suggested that perhaps $\pi_f(G)<2$ when $G$ has girth at
least $5$.  However, Alon~\cite{Alon} observed using the Expander Mixing Lemma
that expander graphs with large girth (such as Ramanujan graphs) still have
$\pi_f$ arbitrarily close to $3$.

Lubotzky, Phillips, and Sarnak~\cite{LPS} introduced {\it Ramanujan graphs}
as $d$-regular graphs in which every eigenvalue with magnitude less than $d$
has magnitude at most $2\sqrt{d-1}$.  For $d-1$ being prime, they further
introduced an infinite family of such graphs whose girth is at least
$\FR23\log_{d-1} n$ when $n$ is the number of vertices.

Let $G$ be a $d$-regular $n$-vertex graph whose eigenvalues other than $d$ have
magnitude at most $\lambda$.  The Expander Mixing Lemma of Alon and
Chung~\cite{AC} states that whenever $A$ and $B$ are two vertex sets in $G$,
the number of edges of $G$ joining $A$ and $B$ differs from $|A|\,|B|(d/n)$ by
at most $\lambda\sqrt{|A|\,|B|}$ (edges with both endpoints in $A\cap B$ are
counted twice).

Alon applied this lemma to an arbitrary vertex ordering $\sigma$ of $G$,
breaking $\sigma$ into $k$ blocks of consecutive vertices, each with length
at most $\CL{n/k}$.  Intuitively, by the Expander Mixing Lemma the vast
majority of the edges can be viewed as forming a blowup of a complete graph
with $k$ vertices.  With $k$ chosen to be about $d^{1/3}$, Alon shows that
asymptotically only $\FR{d^2n^2}{24}$ pairs of nonincident edges can be
separated by $\sigma$ .  However, there are asymptotically $\FR{d^2n^2}8$ pairs
of nonincident edges.  Thus every ordering can separate only about a third of
the pairs.  As noted, this graph $G$ can be chosen to have arbitrarily large
girth.

Alon extended the question in our Conjecture~\ref{conjecture:tripartite} by
asking how small $\epsilon$ can be made so that there is an $n$-vertex graph
$G$ with girth at least $g$ such that $\pi_f(G)\ge 3-\epsilon$.  His detailed
computations~\cite{Alon} with the error terms yield $\epsilon<n^{-c/g}$
for some positive constant $c$.

Graphs with good expansion properties are not planar.
The original paper~\cite{BCGMR} proved $\pi(G)\le2$ for every outerplanar
graph $G$, and hence also $\pi_f(G)\le2$.  Equality holds for outerplanar
graphs with $4$-cycles.  We suggest seeking sharp upper bounds for the family
of outerplanar graphs with girth at least $g$, and similarly for planar graphs
with girth at least $g$.  For the latter question, we suggest first studying
grids (cartesian products of two paths).

\section{Trees} \label{sec:trees}

Although $\lim_{g\to\infty}\FR g{g-2}=1$, it is not true that $\pi_f(G)=1$
whenever $G$ is a tree.  The graphs $G$ with $\pi_f(G)=1$ are just the graphs
with $\pi(G)=1$, as holds for every hypergraph covering parameter.
These graphs were characterized in BCGMR~\cite{BCGMR}.  Each component
is obtained from a path $P$ by adding independent vertices that have one
neighbor or two consecutive neighbors on $P$, but for any two consecutive
vertices on $P$ at most one common neighbor can be added.

This implies that the trees with fractional separation dimension $1$ are the
caterpillars.  We seek the sharpest general upper bound for trees.

\begin{theorem} \label{theorem:treebound}
$\pi_f(G) < \sqrt{2}$ when $G$ is a tree.
\end{theorem}
\begin{proof}
We construct a strategy for the ordering player to show that the separation
game has value at least $\FR1{\sqrt2}$.  Since $\pi_f(G)$ is rational,
the inequality is strict. 

Root $T$ at a vertex $v$.  For a vertex $u$ other than $v$, let $u'$ be the
parent of $u$.  We describe the strategy for the ordering player by an
iterative probabilistic algorithm that generates an ordering.  Starting with
$v$, we iteratively add the children of previously placed vertices according to
the following rules, where $\beta$ is a probability to be specified later.

\eject
\begin{enumerate}[label = (R\arabic*)]
\item The children of $v$ are placed before or after $v$ with probability
$\FR12$, independently.
\item The children of a non-root vertex $u$ or put between $u$ and its parent
$u'$ with probability $1-\beta$; they are place on the side of $u$ away from
$u'$ with probability $\beta$.
\item The children placed on each side of a vertex are placed immediately next
to it by a random permutation.
\end{enumerate}

Note that the resulting ordering has the following property:

\smallskip
\centerline{
$(*)$ Any vertex between a vertex $u$ and a child of $u$ is a descendant of $u$.
}
\smallskip

We must prove that the separation probability is at least $\FR1{\sqrt2}$ for
each pair of nonincident edges.  Given nonincident edges $ab$ and $cd$, let $w$
denote the common ancestor of these vertices that is farthest from the root.
We may assume $a=b'$ and $c=d'$. Without loss of generality, there are three
types of pairs, as shown in Figure~\ref{treefig}.

\begin{figure}[h]
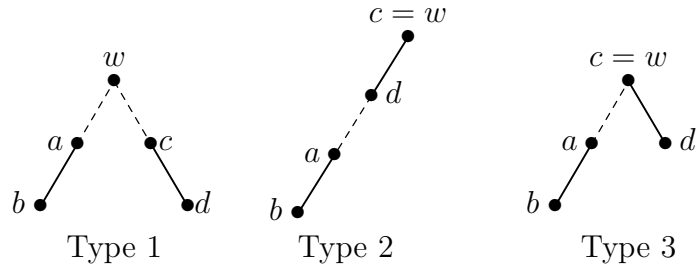

\gpic{
\expandafter\ifx\csname graph\endcsname\relax \csname newbox\endcsname\graph\fi
\expandafter\ifx\csname graphtemp\endcsname\relax \csname newdimen\endcsname\graphtemp\fi
\setbox\graph=\vtop{\vskip 0pt\hbox{%
    \graphtemp=.5ex\advance\graphtemp by 1.000in
    \rlap{\kern 0.115in\lower\graphtemp\hbox to 0pt{\hss $\bu$\hss}}%
    \graphtemp=.5ex\advance\graphtemp by 0.673in
    \rlap{\kern 0.308in\lower\graphtemp\hbox to 0pt{\hss $\bu$\hss}}%
    \graphtemp=.5ex\advance\graphtemp by 0.346in
    \rlap{\kern 0.500in\lower\graphtemp\hbox to 0pt{\hss $\bu$\hss}}%
    \graphtemp=.5ex\advance\graphtemp by 0.673in
    \rlap{\kern 0.692in\lower\graphtemp\hbox to 0pt{\hss $\bu$\hss}}%
    \graphtemp=.5ex\advance\graphtemp by 1.000in
    \rlap{\kern 0.885in\lower\graphtemp\hbox to 0pt{\hss $\bu$\hss}}%
    \special{pn 11}%
    \special{pa 115 1000}%
    \special{pa 308 673}%
    \special{fp}%
    \special{pa 692 673}%
    \special{pa 885 1000}%
    \special{fp}%
    \special{pn 8}%
    \special{pa 308 673}%
    \special{pa 500 346}%
    \special{pa 692 673}%
    \special{da 0.038}%
    \graphtemp=.5ex\advance\graphtemp by 0.231in
    \rlap{\kern 0.500in\lower\graphtemp\hbox to 0pt{\hss $w$\hss}}%
    \graphtemp=.5ex\advance\graphtemp by 1.231in
    \rlap{\kern 0.500in\lower\graphtemp\hbox to 0pt{\hss Type 1\hss}}%
    \graphtemp=.5ex\advance\graphtemp by 1.000in
    \rlap{\kern 0.000in\lower\graphtemp\hbox to 0pt{\hss $b$\hss}}%
    \graphtemp=.5ex\advance\graphtemp by 0.673in
    \rlap{\kern 0.192in\lower\graphtemp\hbox to 0pt{\hss $a$\hss}}%
    \graphtemp=.5ex\advance\graphtemp by 0.673in
    \rlap{\kern 0.769in\lower\graphtemp\hbox to 0pt{\hss $c$\hss}}%
    \graphtemp=.5ex\advance\graphtemp by 1.000in
    \rlap{\kern 0.962in\lower\graphtemp\hbox to 0pt{\hss $d$\hss}}%
    \graphtemp=.5ex\advance\graphtemp by 0.115in
    \rlap{\kern 2.038in\lower\graphtemp\hbox to 0pt{\hss $\bu$\hss}}%
    \graphtemp=.5ex\advance\graphtemp by 0.423in
    \rlap{\kern 1.846in\lower\graphtemp\hbox to 0pt{\hss $\bu$\hss}}%
    \graphtemp=.5ex\advance\graphtemp by 0.731in
    \rlap{\kern 1.654in\lower\graphtemp\hbox to 0pt{\hss $\bu$\hss}}%
    \graphtemp=.5ex\advance\graphtemp by 1.038in
    \rlap{\kern 1.462in\lower\graphtemp\hbox to 0pt{\hss $\bu$\hss}}%
    \special{pn 11}%
    \special{pa 2038 115}%
    \special{pa 1846 423}%
    \special{fp}%
    \special{pa 1654 731}%
    \special{pa 1462 1038}%
    \special{fp}%
    \special{pn 8}%
    \special{pa 1846 423}%
    \special{pa 1654 731}%
    \special{da 0.038}%
    \graphtemp=.5ex\advance\graphtemp by 0.000in
    \rlap{\kern 2.038in\lower\graphtemp\hbox to 0pt{\hss $c=w$\hss}}%
    \graphtemp=.5ex\advance\graphtemp by 1.231in
    \rlap{\kern 1.712in\lower\graphtemp\hbox to 0pt{\hss Type 2\hss}}%
    \graphtemp=.5ex\advance\graphtemp by 0.423in
    \rlap{\kern 1.962in\lower\graphtemp\hbox to 0pt{\hss $d$\hss}}%
    \graphtemp=.5ex\advance\graphtemp by 0.731in
    \rlap{\kern 1.538in\lower\graphtemp\hbox to 0pt{\hss $a$\hss}}%
    \graphtemp=.5ex\advance\graphtemp by 1.038in
    \rlap{\kern 1.346in\lower\graphtemp\hbox to 0pt{\hss $b$\hss}}%
    \graphtemp=.5ex\advance\graphtemp by 1.000in
    \rlap{\kern 2.808in\lower\graphtemp\hbox to 0pt{\hss $\bu$\hss}}%
    \graphtemp=.5ex\advance\graphtemp by 0.673in
    \rlap{\kern 3.000in\lower\graphtemp\hbox to 0pt{\hss $\bu$\hss}}%
    \graphtemp=.5ex\advance\graphtemp by 0.346in
    \rlap{\kern 3.192in\lower\graphtemp\hbox to 0pt{\hss $\bu$\hss}}%
    \graphtemp=.5ex\advance\graphtemp by 0.673in
    \rlap{\kern 3.385in\lower\graphtemp\hbox to 0pt{\hss $\bu$\hss}}%
    \special{pn 11}%
    \special{pa 2808 1000}%
    \special{pa 3000 673}%
    \special{fp}%
    \special{pa 3192 346}%
    \special{pa 3385 673}%
    \special{fp}%
    \special{pn 8}%
    \special{pa 3000 673}%
    \special{pa 3192 346}%
    \special{da 0.038}%
    \graphtemp=.5ex\advance\graphtemp by 0.231in
    \rlap{\kern 3.192in\lower\graphtemp\hbox to 0pt{\hss $c=w$\hss}}%
    \graphtemp=.5ex\advance\graphtemp by 1.231in
    \rlap{\kern 3.192in\lower\graphtemp\hbox to 0pt{\hss Type 3\hss}}%
    \graphtemp=.5ex\advance\graphtemp by 1.000in
    \rlap{\kern 2.692in\lower\graphtemp\hbox to 0pt{\hss $b$\hss}}%
    \graphtemp=.5ex\advance\graphtemp by 0.673in
    \rlap{\kern 2.885in\lower\graphtemp\hbox to 0pt{\hss $a$\hss}}%
    \graphtemp=.5ex\advance\graphtemp by 0.673in
    \rlap{\kern 3.500in\lower\graphtemp\hbox to 0pt{\hss $d$\hss}}%
    \hbox{\vrule depth1.231in width0pt height 0pt}%
    \kern 3.500in
  }%
}%
}
\caption{The three possible types of pairs.\label{treefig}}
\end{figure}

In Type 1, neither edge contains an ancestor of a vertex in the other edge.
Hence $(*)$ implies that no vertex of one edge can lie between the vertices
of the other edge.  Thus Type 1 pairs are separated with probability $1$.

In Type 2, both vertices in one edge are descendants of the vertices in the
other edge, say $a$ and $b$ below $d$.  By $(*)$, the pair fails to be
separated if and only if $a$ is between $c$ and $d$.  This occurs if and only
if the child of $d$ on the path from $d$ to $a$ is placed between $d$ and its
parent, $c$.  This occurs with probability $1-\beta$, so the separation
probability is $\beta$.

In Type 3, the vertices in $ab$ are below $c$ but not $d$.  Again separation
fails if and only if $a$ is between $c$ and $d$.  This requires $d$ and the
child of $c$ on the path to $a$ to be placed on the same side of $c$, after
which the probability of having $a$ between $c$ and $d$ is $\FR12$.  The
probability of having two specified children of $c$ on the same side of $c$ is
$(1-\beta)^2+\beta^2$ if $c\ne v$; it is $\FR12$ if $c=v$.
If $c=v$, then the separation probability is $\FR34$, greater than
$\FR1{\sqrt2}$.  If $c\ne v$, then the separation probability is 
$1-\FR12(1-\beta)^2-\FR12\beta^2$.

We optimize by solving $\beta=1-\FR12(1-\beta)^2-\FR12\beta^2$ and setting
$\beta=\FR1{\sqrt2}$.  Now each pair of nonincident edges is separated with
probability at least $\FR1{\sqrt2}$.
\end{proof}

If a root $v$ can be chosen in a tree $G$ so that the all pairs of Type 3
involve $v$, then in the proof of Theorem~\ref{theorem:treebound} setting
$\beta=\FR34$ yields $\pi_f(G)\le\FR 43$.  This proves the following corollary.  

\begin{cor} \label{cor:spidertoes}
$\pi_f(G) \le \FR43$ for any tree $G$ produced from a subdivision of a star
by adding any number of pendent vertices to each leaf.
\end{cor}

The bound in Corollary~\ref{cor:spidertoes} cannot be improved.

\begin{prop} \label{prop:subdivided stars}
${\pi_f(K'_{1,n}) = \frac{4m-2}{3m-1}}$, where $m=\CL{n/2}$ and $K'_{1,n}$ is
the graph obtained from $K_{1,n}$ by subdividing every edge once.
\end{prop}
\begin{proof}
Form $K'_{1,n}$ from the star with center $v$ and leaves $\VEC y1n$
by introducing $x_i$ to subdivide $vy_i$, for $1\le i\le n$.
Let $X=\VEC x1n$.

If in some ordering a vertex of degree $1$ does not appear next to its
neighbor, then moving it next to its neighbor does not make any separated pair
unseparated.  Hence the ordering player should play only orderings in which
every vertex of degree $1$ appears next to its neighbor; it does not matter on
which side of its neighbor the vertex is placed.

Nonincident edges of the form $x_iy_i$ and $x_jy_j$ are always separated by any
ordering that puts $y_i$ next to $x_i$ for all $i$; the pair player will not
play these.  The remaining $n(n-1)$ pairs of nonincident edges have the form
$\{vx_i,x_jy_j\}$ and lie in a single orbit.  By Corollary~\ref{trans},
some optimal strategy for the pair player makes them equally likely.

An optimal strategy for the ordering player will thus make equally likely all
orderings obtained by permuting the positions of the pairs $x_ry_r$ within an
ordering that maximizes the number of separated pairs of the nontrivial form
$\{vx_i,x_jy_j\}$.  Such a pair is separated when $x_i$ and $x_j$ lie on
opposite sides of $v$ and when $x_i$ is between $v$ and $x_j$.  

To count such pairs, it matters only how many vertices of $X$ appear to the
left of $v$, since $y_i$ appears next to $x_i$ for all $i$.  If $k$ vertices
of $X$ appear to the left of $v$, then the count of separated nontrivial pairs
is $2k(n-k)+\CH k2+\CH{n-k}2$.  This formula simplifies to $\CH{n}2+k(n-k)$,
which is maximized only when $k\in\{\CL{\FR n2},\FL{\FR n2}\}$.

Thus the ordering player puts $v$ in the middle, $\FL{\FR n2}$ vertices of $X$
on one side, and $\CL{\FR n2}$ vertices of $X$ on the other side.
Whether $n$ is $2m$ or $2m-1$, the ratio of $\CH n2+\FL{\FR{n^2}4}$ to
$n(n-1)$ simplifies to $\FR{3m-1}{4m-2}$, as desired.
\end{proof}

\section{Complete Multipartite Graphs}\label{multip}

To prove that always $\pi_f(\Knnt)=\FR{3m}{m+1}$, where $m=\FL{n/2}$,
we need also to compute $\pi_f(K_{m+1,m})$.  This is the special case $q=1$
of our next theorem.  We postponed it because the counting argument for the
generalization is more technical than our earlier arguments.

\begin{theorem}\label{theorem:complete bipartite}
$\pi_f(K_{m+1,qm})=3\left(1-\FR{(q+1)m-2}{(2m+1)mq-m-2}\right)$
for $m,q\in\NN$ with $mq>1$.
\end{theorem}

\begin{proof}
Note that
$3\left(1-\FR{(q+1)m-2}{(2m+1)mq-m-2}\right)=\FR{6m(mq-1)}{(2m+1)mq-m-2}$.
Let $p=\FR{(2m+1)mq-m-2}{6m(mq-1)}$.  
The pairs in $\p$ all lie in the same orbit, so Corollary~\ref{trans}
applies, and the pair player can make all $2\CH{m+1}2\CH{mq}2$ pairs in $\p$
equally likely.  It suffices to show that the maximum number of pairs separated
by any ordering is $2p\CH{m+1}2\CH{mq}2$.  The proof is similar to that of
Theorem~\ref{theorem:balanced bipartite}.

Let the parts of $K_{m+1,qm}$ be $X$ and $Y$, with $|X| = m+1$ and $|Y| =qm$. 
Let $\sigma$ be an ordering $\VEC v0{(q+1)m}$ such that $v_i\in X$ if and only
if $i\equiv 0\mod(q+1)$.  The ordering player will in fact make all such
orderings equally likely.  We show that $\sigma$ separates
$2p\CH{m+1}2\CH{mq}2$ pairs and that no ordering separates more.

Let $X=\{\VEC x0m\}$, indexed in order of appearance in $\sigma$, and similarly
let $Y=\{\VEC y1{qm}\}$.  Let $B_0=\{x_0\}$, and for $1\le i\le m$ let $B_i$
consist of $\{y_{q(i-1)+1},\ldots,y_{qi},x_i\}$.  To count pairs in $\p$
separated by $\sigma$, we consider which blocks contain the vertices used.

If the indices are $i,j,k,l$ with $1\le i < j < k < l \le m$, then one edge
consists of $x_i$ or $x_j$ and a $Y$-vertex from the other block among
$\{B_i,B_j\}$, and similarly for $\{B_k,B_l\}$.  Hence there are
$4q^2\binom{m}{4}$ such pairs.  If $i=0$, then we must use $x_0$, and there are
$2q^2\CH m3$ such pairs.

If the indices are $i,j,k$ with $1\le i<j<k\le m$, then we use two vertices
from one block.  If we use two in $B_i$, then other edge uses $x_j$ or $x_k$
and a $Y$-vertex from the remaining block, yielding $2 q^2$ separated pairs.
Similarly, $2q^2$ separated pairs use two vertices in $B_k$.  Two vertices
used from $B_j$ may both be from $Y$ or may include $x_j$.  In the first case
$x_i$ and $x_k$ are used, while in the second case $x_j$ and $x_i$ are used;
thus the vertices from $Y$ can be chosen in $\CH q2+q^2$ ways.  Hence for such
index choices a total of $(\CH q2+5q^2)\CH m3$ pairs are separated.

If the indices are $0,j,k$ with $1\le j<k\le m$, then $x_i$ is used.
If $x_j$ is used, then there are $q^2$ ways to complete the pair of separated
edges, and if $x_k$ is used then there are $q^2+\CH q2$ ways to complete it.
Hence this case contributes $(\CH q2+2q^2)\CH m2$ pairs.

If the indices are $i$ and $j$ with $1\le i < j\le m$, then either we use two
vertices from each of $B_i$ and $B_j$ (with one edge within each block) or
we use three vertices from $B_j$ and only the vertex $x_i$ from $B_i$.  This
yields $(q^2+ \binom{q}{2}) \binom{m}{2}$ separated pairs.  If $i=0$, then we
must use $x_0$ and three vertices from $B_j$, for a total of $\binom{q}{2} m$
pairs. 

Thus $\sigma$ separates
$[4q^2]\CH m4+[7q^2+\CH q2]\CH m3+[3q^2+2\CH q2]\CH m2+\CH q2m$ pairs.
Direct computation shows that this equals $2p\CH{m+1}2\CH{mq}2$.
In particular, since $p=\FR{(2m+1)mq-m-2}{6m(mq-1)}$, the formula
$2p\CH{m+1}2\CH{mq}2$ simplifies to $\FR1{12}[(2m+1)mq-m-2](m+1)mq$,
and indeed factoring $\FR1{12}mq$ out of the number of pairs separated leaves
$[(2m+1)mq-m-2](m+1)$.

It remains to show that no ordering separates more pairs than the orderings of
this type.  Let $\sigma$ be an ordering not of this type.  Index $X$ and $Y$ as
before.  If $\sigma$ does not start with $x_0$, then let $y$ be the vertex
immediately preceeding $x_0$.  Form $\sigma'$ from $\sigma$ by exchanging the
positions of $y$ and $x_0$.  Since no pair of the form $x'y,x_0y'$ is separated
by $\sigma$, every pair separated by $\sigma$ is also separated by $\sigma'$.

Hence we may assume by symmetry that $\sigma$ starts with $x_0$ and ends with
$x_m$.  If $\sigma$ does not have the desired form, then by symmetry there is
a least index $j$ such that more than $qj$ vertices of $Y$ precede $x_j$, while
fewer than $q(m-j)$ follow $x_j$.  Form $\sigma'$ by exchanging the positions
of $x_j$ and the vertex $y$ immediately preceding it in $\sigma$.
Let $r$ be the number of vertices of $Y$ preceding $x_j$
The number of pairs separated by $\sigma$ but not $\sigma'$ is $j(mq-r)$,
while the number separated by $\sigma'$ but not $\sigma$ is $(r-1)(m-j)$.
The difference is $m(r-jq)-(m-j)$.  Since $r>qj$ and $j<m$, the difference
is positive, and $\sigma'$ separates more pairs than $\sigma$.
\end{proof}

Since $\FR1p=\FR{6(mq-1)}{(2m+1)q-m-2/m}\le \FR{6m}{2m+1}=3(1-\FR1{2m+1})$,
with equality only when $m=1$, always $\pi_f(K_{m,r})$ is bounded away from $3$
by a function of $m$.  In particular, having $\pi_f$ tend to $3$ on a sequence
of bipartite graphs requires the sizes of both parts to grow.

We expect that $\Knnt$ maximizes $\pi_f$ among $n$-vertex bipartite graphs.
By monotonicity, the maximum occurs at $K_{k,n-k}$ for some $k$.  We have
$\pi_f(K_{m,m})=\pi_f(K_{m+1,m})=\FR{3m}{m+1}$.  For unbalanced instances with
$2m+1$ vertices (assuming integrality of ratios for simplicity), 
Theorem~\ref{theorem:complete bipartite} yields
$\pi_f(K_{\FR{2m}{q+1}+1,\FR{2qm}{q+1}})$.  The
value is highest for the balanced case.

It would be desirable to have a direct argument showing that moving a vertex
from the larger part to the smaller part in $K_{k,n-k}$ increases
$\pi_f$ when $k\le n/2-1$.  This would prove
Conjecture~\ref{conjecture:bipartite}.  However, the statement surprisingly
is not true in general for complete tripartite graphs.  Computation has shown
$\pi_f(K_{m+2,m,m})>\pi_f(K_{m+1,m+1,m})$ when $2\le m\le 4$.  Even more
surprising, by computing the values of $\pi_f$ for $K_{m,m,m}$ and $K_{1,m,m}$,
we obtain $\pi_f(K_{1,(n-1)/2,(n-1)/2})>\pi_f(K_{n/3,n/3,n/3})$ when $n$ is an
odd multiple of $3$.  This follows from the remaining results in this section
and motivates our Conjecture~\ref{conjecture:tripartite}.

\begin{theorem}\label{theorem:balanced tripartite}\label{kmmm}
$\pi_f(K_{m,m,m}) = \FR{6m}{2m+1}$ for $m \ge 2$.
\end{theorem}

\begin{proof}
There are two types of pairs in $\p$: those with endpoints in two parts, called
\emph{double-pairs} or \emph{D-pairs}, and those with endpoints in all three
parts, called \emph{triple-pairs} or \emph{T-pairs}.  Within these two types of
pairs in $\p$, any pair can be mapped to any other pair via an automorphism, so
by Corollary~\ref{trans} some optimal strategy for the pair player makes
D-pairs equally likely and makes T-pairs equally likely.  

Let the parts of $K_{m,m,m}$ be $X$, $Y$, and $Z$.  Let $\sigma$ be an ordering
$v_1,\ldots,v_{3m}$ such that each triple $\{v_{3i-2},v_{3i-1},v_{3i}\}$
consists of one vertex from each part, for $1 \le i \le m$.
The ordering player will make all such orderings equally likely.
The restrictions of such orderings to two parts are the orderings used in
Theorem~\ref{theorem:balanced bipartite}, which separate the fraction
$\FR{m+1}{3m}$ of the D-pairs.

We show that each such ordering separates the fraction $\FR{2m+1}{6m}$ of the
T-pairs.  This fraction is smaller than $\FR{m+1}{3m}$.  Hence this strategy
shows that the separation game has value at least $\FR{2m+1}{6m}$.  By making
the T-pairs equally likely, the pair player establishes equality if also no
other ordering separates more T-pairs.

For use in the next theorem, we distinguish each T-pair as a $W$-pair, for
$W\in\{X,Y,Z\}$, when $W$ is the part contributing two vertices to the pair.
Furthermore, with $w\in\{x,y,z\}$, we index $W$ as $\VEC w1m$ in order of 
appearance in $\sigma$.  Let the {\it block} $B_i$ be
$\{v_{3i-2},v_{3i-1},v_{3i}\}$, so $B_i=\{x_i,y_i,z_i\}$ for $1\le i\le m$,
though $\{x_i,y_i,z_i\}$ may appear in any order in $\sigma$.  The vertices of
a T-pair separated by $\sigma$ may use two, three, or four indices in
$\{1,\dots,m\}$.  In each case, let $W$ be the part contributing a vertex to
each edge of the pair.

T-pairs hitting $B_i,B_j,B_k,B_l$ with $i<j<k<l$ consist of one edge in
$B_i\cup B_j$ and the other in $B_k\cup B_l$.  We can choose the blocks for the
two vertices of $W$ in four ways ($B_i$ or $B_j$, and $B_j$ or $B_k$), and then
we just choose which of the other two parts finishes the first edge.  Hence
there are $8\CH m4$ such $W$-pairs for each $W$, together $24\CH m4$ such
T-pairs.

Pairs hitting only $B_i,B_j,B_k$ with $i<j<k$ consist of one edge in
$B_i\cup B_j$ and the other in $B_j\cup B_k$, with care in making a separated
pair needed when $B_j$ contributes two vertices.  If the repeated part $W$
contributes $w_i$ and $w_k$, then there are five ways to complete the pair, one
with two vertices from $B_j$ and two each having an edge within $B_i$ or $B_j$.
If $w_i$ and $w_j$ are used, then there are two $W$-pairs having an edge within
$B_i$.  The number of $W$-pairs using a second vertex from $B_j$ is $t-1$ when
$w_j$ is the $t$th vertex of $B_j$ in $\sigma$, and none having two vertices
from $B_k$.  When $w_j$ and $w_k$ are used, the contributions depending on
the position of $w_j$ are reversed.  Hence in each case 11 $W$-pairs are
separated, for a total of $33\CH m3$ T-pairs.


A separated T-pair hitting only $B_i$ and $B_j$ uses $w_i$ and $w_j$.
Picking one additional vertex from each block yields two separated $W$-pairs.
There is also one separated $W$-pair with three vertices in $B_i$ if $w_i$ is
not the last vertex of $B_i$, and one with three vertices in $B_j$ if $w_j$ is
not the first vertex of $B_j$.  Summing over $W$ yields $10\CH m2$ T-pairs of
this type.


In total $\sigma$ separates $24\CH m4+33\CH m3+10\CH m2$ T-pairs.  This sum
equals $\FR{2m+1}m\CH m2$.  Altogether there are $6m^2\CH m2$ T-pairs,
so the fraction of them separated is $\FR{2m+1}{6m}$, as desired.

For an ordering $\sigma$ not of the specified form, index the vertices of each
part in increasing order in $\sigma$.  Avoiding the specified form means that
some vertex precedes another vertex with a lesser index.  By symmetry, we may
assume that $y_j$ appears immediately before $x_i$ in $\sigma$ for some $i$ and
$j$ with $j > i$.  Let $k$ be the number of vertices of $Z$ appearing before
$y_j$.

Form $\sigma'$ from $\sigma$ by interchanging the positions of $y_j$ and $x_i$.
Any T-pair separated by exactly one of $\sigma$ and $\sigma'$ has $x_i$ and
$y_j$ as endpoints of the two distinct edges.  Considering whether a vertex of
$Z$ is used to complete the first, second, or both edges, there are
$k(m-j)+(i-1)(m-k)+k(m-k)$ such T-pairs in $\sigma$ and
$k(m-i)+(j-1)(m-k)+k(m-k)$ such T-pairs in $\sigma'$.
The difference is $m(j-i)$.  Since $m\ge2$ and $j>i$, the comparison
shows that $\sigma'$ separates strictly more T-pairs than $\sigma$.
\end{proof}

We also compute the fractional separation dimension of $K_{m+1,m,m}$.
As with $K_{m+1,m}$, the extra vertex imposes no extra cost.

\begin{theorem}\label{theorem:mod1 tripartite}\label{km1mm}
$\pi_f(K_{m+1,m,m}) = \FR{6m}{2m+1}$ for $m \ge 2$.
\end{theorem}

\begin{proof}
Let the parts of $K_{m+1,m,m}$ be $X$, $Y$, and $Z$ with $|X| = m+1$.  By
monotonicity, $\pi_f(K_{m+1,m,m}) \ge \pi_f(K_{m,m,m}) = \FR{6m}{2m+1}$.
To prove equality, it suffices to give a strategy for the ordering player that
separates any pair in $\p$ with probability at least $\FR{2m+1}{6m}$.
Given an ordering $\sigma$ as $\VEC v1{3m+1}$, let
$B_i=\{v_{3i-2},v_{3i-1},v_{3i}\}$ for $1\le i\le m$ as in Theorem~\ref{kmmm}.
Use $W\in\{X,Y,Z\}$ and $W=\{\VEC w1t\}$ as before, indexed as ordered in
$\sigma$.  The ordering player makes equally likely all orderings such that
$(v_{3i-2},v_{3i-1},v_{3i})=(x_i,y_i,z_i)$ in order, with $x_{3m+1}$ at the
end, and all those that switch $Y$ and $Z$.  By Corollary~\ref{trans}, it
suffices to show that $\sigma$ separates at least the fraction $\FR{2m+1}{6m}$
of the pairs in each orbit.

For the pairs in $\p$ with endpoints in only two parts, the number of pairs
separated by $\sigma$ depends only on the restriction of $\sigma$ to those
parts.  The restriction is precisely an ordering used in
Theorem~\ref{theorem:balanced bipartite} or Theorem~\ref{theorem:complete
bipartite}.  There we showed that the fraction of such pairs separated is
$\frac{m+1}{3m}$, which is larger than $\FR{2m+1}{6m}$.

In remains to consider the T-pairs.  As in Theorem~\ref{kmmm}, classify these
as $W$-pairs for $W\in\{X,Y,Z\}$.  The $Y$-pairs and $Z$-pairs are in one
orbit, the $X$-pairs in another.

Deleting $x_{3m+1}$ (the last vertex) leaves an ordering considered in
Theorem~\ref{kmmm}.  There we counted $W$-pairs within that ordering.  There
were the same number of separated T-pairs of each type, except for those
hitting only two blocks.  Since each block $B_k$ appears in the order
$(x_k,y_k,z_k)$, each pair of blocks yields three separated $X$-pairs,
four $Y$-pairs, and three $Z$-pairs among the 10 $T$-pairs counted earlier.

We conclude that the ordering separates
$8\CH m4+11\CH m3+3\CH m2$ $X$-pairs and a total of
$16\CH m4+22\CH m3+7\CH m2$ $Y$-pairs and $Z$-pairs not involving $x_{3m+1}$.

Separated T-pairs involving $x_{3m+1}$ hit at most three earlier blocks.
Using one vertex each from $B_i$, $B_j$, and $B_k$ with $i<j<k$, we obtain 
$4\CH m3$ $X$-pairs and a total of $4\CH m3$ $Y$-pairs and $Z$-pairs.
Using $x_{3m+1}$ and vertices from $B_i$ and $B_j$, there
are $5\CH m2$ $X$-pairs, $2\CH m2$ $Y$-pairs and $\CH m2$ $Z$-pairs.  Using
$x_{3m+1}$ and all three vertices of $B_i$, we obtain one $X$-pair, since $x_i$
comes first, and no $Y$-pairs or $Z$-pairs.

%

%

Summing these possibilities, we find that $\sigma$ separates
$8\CH m4+15\CH m3+8\CH m2+m$ of the $m^3(m+1)$ $X$-pairs and
$16\CH m4 + 26\CH m3+10\CH m2$ of the $2m^2(m^2-1)$ $Y$-pairs and $Z$-pairs. 
Remarkably, each ratio is exactly $\FR{2m+1}{6m}$.
\end{proof}

In Theorem~\ref{kmmm} we used more general orderings to simplify the
optimality argument.  Not needing that proof, here we used more restricted
orderings to simplify counting T-pairs.

\begin{theorem}\label{k1mm}
$\pi_f(K_{1,m,m}) = \FR{24m}{8m+5+3/(2\CL{m/2}-1)}$ for $m\ge1$.
\end{theorem}

\begin{proof}
Let the parts be $X$, $Y$, and $Z$ with $X=\{x\}$.  Again we have D-pairs and
T-pairs, but the D-pairs all lie in $Y\cup Z$, and the T-pairs all use $x$
and are $Y$-pairs or $Z$-pairs, designated by the part contributing a vertex
to each edge.  The D-pairs lie in one orbit, as do the T-pairs, so by
Corollary~\ref{trans} some optimal strategy for the pair player makes
D-pairs equally likely and makes T-pairs equally likely.

Let $\sigma$ be a vertex ordering of the form $\VEC v1{2k},x,\VEC v{2k+1}{2m}$
such that each pair of the form $\{v_{2i-1},v_{2i}\}$ consists of one vertex
from each of $Y$ and $Z$, for $1 \le i \le m$.
We count the pairs separated by $\sigma$.  After optimizing over $k$, the
ordering player will make all orderings with that $k$ equally likely.

For all $k$, the restrictions of such orderings to $Y\cup Z$ are the orderings
used in Theorem~\ref{theorem:balanced bipartite}, which separate the fraction
$\FR{m+1}{3m}$ of the D-pairs, and no ordering separates more such pairs.

Index $Y$ as $\VEC y1m$ and $Z$ as $\VEC z1m$ in order in $\sigma$, so that
$\{v_{2i-1},v_{2i}\}=\{y_i,z_i\}$ for $1\le i\le m$.  Each T-pair separated by
$\sigma$ involves $x$.  For the edge $xw$, an edge separated from $xw$ by
the ordering is obtained by picking one vertex each from $Y$ and $Z$ that
are both on the opposite side of $x$ from $w$ or both on the opposite side of 
$w$ from $x$.  When $w\in\{y_j,z_j\}$ with $1\le j\le k$, taking the two
cases of $y_j$ and $z_j$ together yields $(j-1)(j-1+j)+2(m-k)^2$ pairs.
Summing over $j$ yields $2k(m-k)^2+\SE j1k \CH{2j-1}2$ pairs.  Similarly, 
summing over $k+1\le j\le m$ yields $2(m-k)k^2+\SE j1{m-k}\CH{2j-1}2$ pairs.

Let $f(k)$ be the sum of these two quantities, the total number of T-pairs
separated.  Note that $f(k)=2mk(m-k)+\SE j1k\CH{2j-1}2+\SE j1{m-k}\CH{2j-1}2$.
Letting $g(k)=f(k)-f(k-1)$, we have
$g(k)=2m(m-2k+1)+\CH{2k-1}2-\CH{2(m-k+1)-1}2$, which simplifies to
$m-2k+1$.  Thus $g(k)$ is a decreasing function of $k$.
Also, $g(\FR m2)>0$ and $g(\FR{m+1}2)=0$.  Hence the number of T-pairs is
maximized by choosing $k$ as the integer closest to $m/2$.

By induction on $k$, it is easily verified that
$\SE j1k\CH{2j-1}2=\FR16(4k+1)k(k-1)$.  Hence when $m$ is even and $k=m/2$,
our orderings separate $\FR m{12}(8m^2-3m-2)$ pairs.  When $m$ is odd,
they separate $\FR{m-1}{12}(8m^2+5m+3)$.  With altogether $2m^2(m-1)$
T-pairs, the ratio is $\FR{8m^2-3m-2}{24m(m-1)}$ when $m$ is even and
$\FR{8m^2+5m+3}{24m^2}$ when $m$ is odd.  Dividing numerator and denominator by
$m-1$ or $m$ yields the unified formula $\FR{8m+5+3/(2\CL{m/2}-1)}{24m}$ for
the fraction separated.

Note that the fraction of T-pairs separated is smaller than the fraction of
D-pairs separated.  It suffices to show that no ordering that does not
pair vertices of $Y$ and $Z$ and place $x$ between two pairs separates
the maximum number of $T$-pairs.  The pair player achieves equality in the
game by making the T-pairs equally likely.

%
%
%

Since we have considered all $k$, avoiding the specified form means that
some vertex in $Y\cup Z$ precedes another vertex with a lesser index or that
$x$ occurs between $y_i$ and $z_i$ for some $i$.
In the first case, we may assume that $y_j$ appears before $z_i$ with
$j>i$ and no vertex of $Y\cup Z$ between $y_j$ and $z_i$.
In the second case, we may assume by symmetry that $i<m$ and $y_i$ is
before $z_i$.  In either case, form $\sigma'$ from $\sigma$ by moving $z_i$ one
position earlier; this exchanges $z_i$ with $y_j$ or with $x$.

If $x$ appears before $y_j$ in $\sigma$, then $m-j$ T-pairs are separated in
$\sigma$ but not $\sigma'$, and $m-i$ T-pairs are separated in $\sigma'$ but
not $\sigma$.  If $x$ appears after $z_i$, then $i-1$ T-pairs are separated by
$\sigma$ but not $\sigma'$, and $j-1$ T-pairs are separated by $\sigma'$ but
not $\sigma$.  Since $j>i$, in each case $\sigma'$ separates more T-pairs.

In the remaining case, $x$ appears between $y_j$ and $z_i$ with $j\ge i$
and $i<m$.  Now $(i+j-1)(m-j)$ T-pairs are separated by $\sigma$ but not
$\sigma'$, and $(2m-i-j)j$ T-pairs separated by $\sigma'$ but not $\sigma$.
We have $(2m-i-j)j > (i+j-1)(m-j)$ when $j<m(j-i+1)$, which is true
when $i<j\le m$ and $i<m$.
\end{proof}


\section{Circular Separation Dimension} \label{sec:circular}

Instead of considering linear orderings of the $V(G)$, we may consider circular
orderings of $V(G)$.  A pair of nonincident edges $\{xy,zw\}$ is
\emph{separated} by a circular ordering $\sigma$ if the endpoints of the two
edges do not alternate.  The {\it circular separation dimension} is the minimum
number of circular orderings needed to separate all pairs of nonincident edges
in this way.  The \emph{circular $t$-separation dimension} $\pi_t^\circ(G)$ is
the minimum size of a multiset of circular orderings needed to separate all
the pairs at least $t$ times.  The \emph{fractional circular separation
dimension} $\pi_f^\circ(G)$ is $\liminf_{t \to \infty} \pi_f^\circ(G)/t$. 

Like $\pi(G)$, also $\pi^\circ(G)$ is a hypergraph covering problem.
The vertex set $\p$ of the hypergraph $H$ is the same, but the edges
corresponding to vertex orderings of $G$ are larger.
Thus $\pi^\circ(G)\le\pi(G)$ and $\pi_f^\circ(G)\le \pi_f(G)$.

Before discussing the fractional problem, one should first determine the
graphs $G$ such that $\pi^\circ(G)$ (and hence also $\pi_f^{\circ}(G)$)
equals $1$.  Surprisingly, this characterization is quite easy.
Unfortunately, it does not generalize to geometrically characterize graphs
with $\pi^\circ(G)=t$ like the boxicity result in~\cite{BCGMR,BCGMR2}.

\begin{prop}\label{outer}
$\pi^\circ(G)=1$ if and only if $G$ is outerplanar.
\end{prop}
\begin{proof}
When $\pi^\circ(G)=1$, the ordering provides an outerplanar embedding of $G$
by drawing all edges as chords.  Chords cross if and only if their endpoints
alternate in the ordering.

For sufficiency, it suffices to consider a maximal outerplanar graph, since
the parameter is monotone.  The outer boundary in an embedding is a spanning
cycle; use that as the vertex order.  All pairs in $\p$ are separated, since
alternating endpoints yield crossing chords.
\end{proof}

The lower bound $\pi(K_{m,n})\ge\log_2(\min\{m,n\})$ relies on the fact that
when two vertices of one part precede two vertices of the other, both
nonincident pairs induced by these four vertices fail to be separated.  In an
circular ordering, always at last one of the two pairs is separated.  This
leads to the surprising result that $\pi^\circ(G)\in\{1,2\}$ when $G$ is
bipartite.

\begin{prop}\label{bipartc}
$\pi^\circ(K_{m,n})=2$ when $m,n\ge2$ with $mn>4$.
\end{prop}
\begin{proof}
The exceptions are the cases where $K_{m,n}$ is outerplanar and 
Proposition~\ref{outer} applies.  Let $\sigma$ be an circular ordering in which
each partite set occurs as a consecutive segment of vertices.  Obtain $\sigma'$
from $\sigma$ by reversing one of the partite sets.  A nonincident pair of 
edges alternates endpoints in $\sigma$ if and only if it does not alternate
endpoints in $\sigma'$.  Hence it is separated in exactly one of the two
orderings.
\end{proof}

Nevertheless, $\pi^\circ$ is unbounded.  It suffices to consider $K_n$, where a
classical result provides the lower bound.  A list of $d$-tuples is {\it
monotone} if in each coordinate the list is strictly increasing or weakly
decreasing.  The multidimensional generalization of the Erd\H{os}--Szekeres
Theorem by de~Bruijn states that any list of more than $l^{2^d}$ vectors in
$\RR^d$ contains a monotone sublist of more than $l$ vectors.  The result is
sharp, but this does not yield equality in the lower bound on $\pi^\circ(K_n)$.
Our best upper bound is logarithmic, from
$\pi(K_n)\le 4\log_{3/2}n$~\cite{BCGMR2}.

\begin{theorem}
$\pi^\circ(G)>\log_2\log_3(\omega(G)-1)$.
\end{theorem}
\begin{proof}
Note first that a set of circular orderings separates all pairs of 
nonincident edges in $K_n$ if and only if every $4$-set appears cyclically
ordered in more than one way (not counting reversal).  This follows because
each cyclic ordering of $K_4$ alternates endpoints of exactly one pair of
nonincident edges, and for the three cyclic orderings (unchanged under
reversal) the pairs that alternate are distinct.

Consider $d$ circular orderings of $\{\VEC v1n\}$.  Write them linearly by
starting with $v_1$.  Associate with each $v_i$ a vector $w_i$ in $\RR^d$ whose
$j$th coordinate is the position of $v_i$ in the $j$th linear ordering.  If
$n>3^{2^d}$, then by the multidimensional generalization of the
Erd\H{o}s--Szekeres Theorem $\VEC w1n$ has a monotone sublist of four elements.
The four corresponding vertices $x_1,x_2,x_3,x_4$ appear in increasing order or
in decreasing order in each linear order.  Hence they appear in the same cyclic
order or its reverse in each of the original circular orderings.  In
particular, $x_1x_3$ and $x_2x_4$ are not separated by these circular
orderings.  Since we considered any $d$ circular orderings, $\pi^\circ(K_n)>d$
when $n=3^{2^d}+1$.
\end{proof}

We next turn to the fractional context.  Since $\pi^\circ(G)$ is a hypergraph
covering problem, again and $\pi_f^\circ$ is computed from a matrix game, with
each row provided by the set of pairs in $\p$ separated by a circular orderinge

Our earlier results have analogues in the circular context.  A circular
ordering of four vertices separates two of the three pairs instead of one,
which improves some bounds by a factor of $2$.  The characterization of the
extremal graphs then mirrors the proof of Theorem~\ref{theorem:K4free}.

\begin{theorem}\label{circext}
$\pi_f^\circ(G)\le \FR32$, with equality if and only if $K_4\esub G$.
Furthermore, if $G$ has $n$ vertices and $K_4\esub G$,
then $\pi_f^\circ(G)\le\FR32\left(1-\FR6{n^4}+O(\FR1{n^5})\right)$.
\end{theorem}
\begin{proof}
A circular ordering separates two of the three pairs in each set of four
vertices, so making all circular orderings of $n$ vertices equally likely
yields $\pi_f^\circ(G)\le \FR32$.  Equality holds when $K_4\esub G$, since the
pair player can give probability $\FR13$ to each pair of nonincident edges in a
copy of $K_4$.

Now suppose $K_4\not\esub G$.  Let $p=\FR23+\FR{4(n-4)!}{n!}$.  We provide a
distribution on the circular orderings of $V(G)$ such that each nonincident
pair of edges is separated with probability at least $p$.  We create a list of
$n!$ linear orderings of $V(G)$, which we view as $n!$ circular orderings.

Consider $S=\{a,b,c,d\}\esub V(G)$.  For each ordering $\rho$ of the remaining
$n-4$ vertices, $24$ orderings begin with $S$ and end with $\rho$.  By
symmetry, we may assume $ac\notin E(G)$.  Thus the possible pairs of
nonincident edges induced by $S$ are $\{ab,cd\}$ and $\{ad,bc\}$.  We increase
the separation probability for these vertex pairs.

Circular separation includes nesting when written linearly; only alternation of
endpoints fails.  The pairs $\{ab,cd\}$ and $\{ad,bc\}$ are each separated $16$
times in the $24$ orderings of $S$ followed by $\rho$.  The new $24$ orderings
will separate $\{ab,cd\}$ and $\{ad,bc\}$ each at least $20$ times and any
other pair (not including $\{a,c\}$) at least $16$ times.

The $24$ new orderings are two copies each where the first four vertices are
(in order) $abdc$, $badc$, $dcba$, $cbad$, $adbc$, $adcb$, $acbd$, or $dbac$,
and four copies each using $cdab$ or $bcda$, always followed by $\rho$.  By
inspection, each of $\{ab,cd\}$ and $\{ad,bc\}$ is separated $20$ times in the
list.

The number of orderings that separate any pair of nonincident edges having at
most two vertices in $S$ is the same as before.  Hence we need only check pairs
with three vertices in $S$, consisting of one edge in $\{ab,cd,bc,ad,bd\}$
(never $ac$) and another edge with one endpoint among the remaining two
vertices in $S$.  In each case, the endpoints of the induced edge appear before
or after the third vertex in at least $16$ of the orderings in the new list of
$24$.

Since $\{a,b,c,d\}$ is arbitrary and we do this for each $4$-set, the pairs
$\{ab,cd\}$ and $\{ad,bc\}$ are separated with probability at least $\FR56$ by
the $24$ orderings that start with $\{a,b,c,d\}$ and then are made circular,
and with probability at least $\FR23$ among the remaining orderings.  Thus the
separation probability increases from $\FR23$ to at least $p$ for each pair.
\end{proof}

Again there is no sharper bound for bipartite graphs or graphs with girth $4$:
$\pi_f^\circ(K_{m,m})\to\FR 32$.  The orderings used to give the optimal upper
bound for $\pi_f^\circ(K_{m,qm})$ are in some sense the farthest possible from
those giving the optimal upper bound for $\pi^\circ(K_{m,qm})$ in
Proposition~\ref{bipartc}.

\begin{theorem}\label{bipartfc}
$\pi_f^\circ(K_{m,qm}) = \FR{6(qm-1)}{4mq+q-3}$.  In particular,
$\pi_f^\circ(K_{m,m})=\FR{3m-3}{2m-1}$.
\end{theorem}
\begin{proof}
Again Corollary~\ref{trans} (for the circular separation game) applies.  The
$2\CH m2\CH{qm}2$ pairs of nonincident edges lie in one orbit, so it suffices
to make circular orderings that separate $\FR{4mq+q-3}{6(qm-1)}2\CH m2\CH{qm}2$
pairs equally likely and show that no ordering separates more.

Let $X$ and $Y$ be the parts of the bipartition, with $|X|=m$.  Let $\sigma$
be a circular ordering in which the vertices of $X$ are equally spaced, with
$q$ vertices of $Y$ between any two successive vertices of $X$.

There are two types of pairs separated by $\sigma$.  In one, the parts
for the four vertices alternate as $XYXY$; in the other, they occur as
$XYYX$, cyclically.  Choose the first member of $X$ in $m$ ways.  Let $k$ be
the number of steps within $X$ taken to get from there to the other member of
$X$ used.  In the first case, there are $kq(m-k)q$ ways to choose the vertices
from $Y$ and two ways to group the chosen vertices to form a separated
nonincident pair, but either of the vertices of $X$ could have been called the
first vertex.  In the second case, there are $\CH{kq}2$ ways to choose from
$Y$, one way to group, and only one choice for the first vertex of $X$.

Thus to count the separated pairs, we sum over $k$ and use
$\SE k{-n}m\CH{n+k}r\CH{m-k}s=\CH{n+m+1}{r+s+1}$ and
$\SE k1n k^2=\FR16 n(n+1)(2n+1)$ to compute
\begin{align*}
m\SE k1{m-1} kq(m-k)q + m\SE k1{m-1}\CH{kq}2
&=m\SE k0m q^2\CH{0+k}1\CH{m-k}1+\FR{mq}2\SE k1{m-1}(k^2q-k)\\
&=mq^2\CH{m+1}3+\FR{mq^2}2\FR{(m-1)m(2m-1)}6-\FR{mq}2\CH m2.
\end{align*}
Factoring out $2\CH m2\FR{qm}2$ leaves $\FR16(4mq+q-3)$, as desired.

It remains to show that no other circular ordering separates as many pairs of
nonincident edges.  We do this by finding, for every circular ordering $\sigma$
other than those discussed above, an ordering $\hat\sigma$ that separates more
pairs.

With $X=\{\VEC x1m\}$ in cyclic order, the ordering $\sigma$ is
described by a list $\VEC q1m$ of nonnegative integers summing to $qm$, where
$q_i$ is the number of vertices of $Y$ between $x_{i-1}$ and $x_i$ (indexed
modulo $m$).  Index so that $q_1=\max_i q_i$; we may assume $q_1\ge q+1$.

Let $\sigma'$ be the ordering obtained by interchanging $x_m$ with the vertex
$y$ immediately following it (note that $y\in Y$, since $q_1>q$).  The pairs in
$\p$ separated by $\sigma$ or $\sigma'$ but not both are those consisting of an
edge $yx_k$ for some $k$ with $1\le k\le m-1$ and an edge $x_my'$.  For those
separated by $\sigma$ but not $\sigma'$ there are $\SE j{k+1}m q_j$ choices for
$y'$.  For those separated by $\sigma'$ but not $\sigma$ there are
$(\SE i1k q_i)-1$ choices for $y'$.

After isolating the terms involving $q_1$, the net gain in switching from
$\sigma$ to $\sigma'$ is thus
$$\SE k1{m-1}\left(q_1-1+\SE i2k q_i-\SE j{k+1}m q_j\right).$$

Consider instead the ordering $\sigma''$ obtained from $\sigma$ by
interchanging $x_1$ with the vertex $y$ immediately preceding it (again
$y\in Y$, since $q_1>q$.  The net change in the number of separated pairs
follows the same computation, except that $\VEC q2m$ are indexed in the reverse
order.  More precisely, the change in moving from $\sigma$ to $\sigma''$ is
$$\SE k2m\left(q_1-1+\SE j{k+1}m q_j-\SE i2k q_i\right).$$

In summing the two net changes, the summations in the terms for $2\le k\le m-1$
cancel.  The sum is thus
$$
2(q_1-1)(m-1)-\SE j2m q_j-\SE i2m q_i.
$$
Since $\SE j2m q_j=qm-q_1$, the net sum simplifies to $2q_1m-2qm-2(m-1)$.
Since $q_1\ge q+1$, the value is at least $2$.  Since the sum of the two net
changes is positive, at least one of them is positive, and $\sigma$ does
not separate the most pairs.
\end{proof}

Note that $K_{2,r}$ is planar with girth $4$, for $r\ge2$.
Theorem~\ref{bipartfc} yields $\pi_f^\circ(K_{2,2q})=\FR{4q-4}{3q-1}\to \FR43$.
It remains open how large $\pi_f^\circ$ can be for planar graphs with girth $4$,
and for graphs (planar or not) with larger girth.  For girth $5$, computer
computation shows that the fractional circular separation dimension of the
Petersen graph is $\FR87$.

\end{document}